\documentclass[12pt,twoside]{amsart}
\usepackage{mathrsfs}
\usepackage{latexsym}
\usepackage{amssymb}
\usepackage{amsfonts}
\usepackage{amstext}
\usepackage{multicol}
\usepackage{amsmath, amsthm}
\usepackage{setspace}
\usepackage[colorinlistoftodos]{todonotes}
\usepackage{bbm}

\usepackage[utf8]{inputenc}

\newtheorem{theorem}{Theorem}[section]
\newtheorem{lemma}[theorem]{Lemma}
\newtheorem{prop}[theorem]{Proposition}
\newtheorem{cor}[theorem]{Corollary}
\newtheorem{conj}[theorem]{Conjecture}

\thanks{  Partially supported by M. Rudelson's  NSF
 Grant  DMS-1464514,  and USAF Grant FA9550-14-1-0009.}
\begin{document}
\title[John Ellipsoid and the Center of Mass]{John 
Ellipsoid and the center of mass \\ of a convex body}
\date{}
\author{Han Huang}
\maketitle

\begin{abstract}

It is natural to ask whether the center of mass of a 
convex body $K\subset \mathbb{R}^n$ lies in its John 
ellipsoid $B_K$, i.e., in the maximal volume ellipsoid 
contained in $K$. This question is relevant to the 
efficiency of many algorithms for convex bodies. In 
this paper, we obtain an unexpected negative result. 
There exists a convex body $K\subset \mathbb{R}^n$ such 
that its center of mass does not lie in the  John 
ellipsoid $B_K$  inflated $(1-C\sqrt{\frac{\log(n)}
{n}})n$ times about the center of $B_K$. Moreover, 
there exists a polytope $P \subset
\mathbb{R}^n$ with $O(n^2)$ facets whose center of mass 
is not contained in the John ellipsoid  $B_P$ inflated 
$O(\frac{n}{\log(n)})$ times about the center of $B_P$.
\end{abstract}

\begin{flushleft}

\section{Introduction}
Recall that the John ellipsoid $B_K$ of a convex 
body $K\subset \mathbb{R}^n$ is the maximal volume 
ellipsoid contained in $K$. A natural question
asked by  S.~Vempala is whether the center of mass of 
$K$  lies in a small dilation of its John ellipsoid. 
The importance of this question stems from its relation 
to the efficiency of algorithms for convex bodies. The
efficiency of many such algorithms  depends on the 
"roundness" of the body. This can be measured in two 
ways:
\begin{enumerate}

\item the traditional way, as the ratio of the radii of 
the circumscribed to the inscribed ball;

\item as the ratio of the radii of the smallest ball 
that contains the most points (say 1/2 of the volume) to the 
inscribed ball.

\end{enumerate}
   For instance, the complexity of sampling algorithms 
grows quadratically with the latter ratio. Thus, a 
common pre-processing step is to find a good rounding--in
other words, find an ellipsoid for which this 
ratio is reasonably small and then map it to the unit 
ball using an affine transformation. This can be done 
in a randomized polynomial time algorithm by estimating 
the inertia ellipsoid (defined by the covariance matrix 
of a uniform random point from $K$), wherein complexity 
depends logarithmically on the initial ratio of the 
radii, but as a large degree polynomial on the 
dimension. The other possible candidate is the John 
ellipsoid. This ellipsoid is difficult to construct in 
general, but for explicit polytopes, a simple iterative 
algorithm identifies the inscribed ellipsoid of the 
maximal volume quite efficiently. This algorithm was 
developed by L.~G.~Khachiyan \cite{mie}. Recently, 
Y.~Lee and A.~Sidford have provided a faster algorithm 
\cite{Efficient Inverse Maintenance and Faster 
Algorithms for Linear Programming. }. In contrast to 
the inertia ellipsoid, whose construction requires 
sampling, the John ellipsoid is constructed 
deterministically. The John ellipsoid can be used to 
reduce the ratio (1) but it can be as large as $n$, 
which is the dimension of the body. On the other hand, 
the inertia ellipsoid yields the bound $O(\sqrt{n})$ 
for the ratio (2). This raises a question: Does the 
John ellipsoid also provide a good bound for the ratio 
(2)? In other words, one can write it as the following 
conjecture:
   
\begin{conj} \label{conj1}
	For any convex body $K$ in $\mathbb{R}^n$, the John 
	ellipsoid of $K$ scaled by a factor of $O(\sqrt{n})$ 
	about the ellipsoid's center will contain at least 
	half of the volume of $K$.
\end{conj}

This can be formulated in terms of the center of mass. We 
will show in Section 4 that Conjecture \ref{conj1} is 
equivalent  to the following conjecture:

\begin{conj} \label{conj2}
	For any convex body $K$ in $\mathbb{R}^n$, the John 
	ellipsoid of $K$ scaled by a factor of $O(\sqrt{n})$ 
	about the ellipsoid's center will contain the center 
	of mass of $K$.
\end{conj}

The main result of this paper is:


\begin{theorem} \label{main theorem}

	For a sufficiently large $n \in \mathbb{N}$,
    \begin{enumerate}
    \item There exists a convex body 
    $K\subset\mathbb{R}^n$ such that its center of mass 
    does not lie in the John ellipsoid scaled by a 
    factor of $(1-C_0\sqrt{\frac{\log(n)}{n}})n$ about 
    the ellipsoid's center, where $C_0>0$ is a 
    universal constant.
    
    \item There exists a polytope $P\subset \mathbb{R}
	 ^n$ with $O(n^2)$ facets such that its center of 
	 mass does not lie in the John ellipsoid scaled by a 
	 factor of  $C_1\frac{n}{\log(n)}$ about the 
	 ellipsoid's center, where $C_1>0$ is a universal 
	 constant.
    \end{enumerate}
\end{theorem}

Remark: It is well known that for any convex body 
$K\subset \mathbb{R}^n$, the John ellipsoid of $K$
scaled by a factor $n$ about the ellipsoid's center 
contains the original body $K$. (\cite{John}) 

Thus, the example in Theorem \ref{main theorem}(1) is 
the asymptotically optimal in the sense that $
\lim_{n\rightarrow +\infty}\frac{(1-
C_0\sqrt{\frac{\log(n)}{n}})n}{n}=1$.

A consequence of this theorem is the following:
\begin{cor} \label{main cor}
	For a sufficiently large $n \in \mathbb{N}$,
    \begin{enumerate}
    
    \item There exists a convex body $K\subset 
    \mathbb{R}^n$ such that the center of its John 
    ellipsoid $B_K$ is $0$ and $$\text{vol}((1-
    C_0'\sqrt{\frac{\log(n)}{n}})n B_K\cap K) \le 
    \frac{1}{2}\mbox{vol}(K),$$ where $C_0'>0$ is a 
    universal constant.
    
    \item There exists a polytope $P\subset \mathbb{R}
    ^n$ with $O(n^2)$ facets such that the center of 
    its John ellipsoid $B_P$ is $0$ and  $$\mbox{vol}
    (C_1'\frac{n}{\log(n)}B_P\cap P)\le \frac{1}
    {2}\mbox{vol}(P),$$ where $C_1'>0$ is a universal 
    constant.

\end{enumerate}
\end{cor}

Thus, Conjecture \ref{conj1} and Conjecture \ref{conj2}
are not true due to Theorem \ref{main theorem} and 
Corollary \ref{main cor}. In particular, both 
conjectures will not hold even if one restricts the 
collection of convex bodies to polytopes with $O(n^2)$ 
facets.\\

This paper is structured as follows. Section 2 examines 
the notation and necessary background for the proof of 
the main theorem. The proof of the main theorem is
presented in Section 3. Corollary \ref{main cor} and 
the relation between Conjecture \ref{conj1} and 
\ref{conj2} are examined in Section 4.


\section{Notations and Prelimimaries}

Let $B_2^n$ denote the unit Euclidean ball in $
\mathbb{R}^n$ and $|\cdot|$ denote the Euclidean norm. 
Let $\{e_i\}_{i=1}^n$ be the standard orthonormal basis 
for $\mathbb{R}^n$. For any $x\in \mathbb{R}^n$, let 
$x_i$ denote its i-th coefficient.\\

A subset of $\mathbb{R}^n$ is called a convex body if 
it is a convex, compact set that has a non-empty 
interior. For a subset $A\subset \mathbb{R}^n$, let 
$1_A$ denote the indicator function of $A$. 
  
For a convex body $K\subset \mathbb{R}^n$, $\text{vol}
(K):=\int_{\mathbb{R}^n} 1_{K}(x)dx$, where the 
integral is the standard Lebesgue integration on $\mathbb{R}^n$. \\

An ellipsoid $B_K$ is the John's ellipsoid of $K$ if 
$B_K\subset K$; for any other ellipsoid $
\mathcal{E}\subset K$, $\text{vol}(B_K)\ge\text{vol}
(\mathcal{E})$. It is known that $B_K$ exists and is 
unique.(\cite{John})

A convex body $K$ is in John's position if the John 
ellipsoid of $K$ is $B_2^n$. For any convex body 
$K\subset \mathbb{R}^n$, there exists an affine 
transformation $T$ such that $TK$ is in John's 
position.

If a convex body $K$ contains $0$, we can define its 
radial function $\rho: S^{n-1} \rightarrow \mathbb{R}_+
$ by

$$
  \forall \theta \in S^{n-1} \, , \, \rho(\theta)= 
  \max \{ t>0 \, , \,t\theta \in K\}.
$$ 
  
The center of mass of a convex body $K$ is defined by 
$x_K:=\frac{1}{\text{vol}(K)}\int_{K}xdx$.  
  
Let $\sigma_{n-1}$ denote the normalized 
Haar measure on $S^{n-1}$. An absolutely continuous 
measure $\mu$ with density function $\frac{d\mu(x)}{dx}
=f(x)$ on $\mathbb{R}^n$ is called log-concave if, for any $\lambda \in (0,1)$ and $x,y \in 
\mathbb{R}^n$, we have

\begin{equation}
\label{logconcave}
 	f(\lambda x+(1-\lambda) y)\ge f(x)^{\lambda}f(y)^{1-
 	\lambda}.
\end{equation}

In this paper, we are interested in a specific class of 
log-concave measures: For any convex body $K$, the 
probability uniformly distributed in $K$,
$\frac{1_{K}}{\text{vol}(K)}dx$, is a log-concave 
probability measure. This is due to the fact that indicator
functions of convex sets automatically satisfied the inequality 
(\ref{logconcave}).\\

For matrices, let $\text{Tr}(M)$ denote the trace of a 
square matrix $M$ and $I_n$ denote the identity matrix 
on $\mathbb{R}^n$.\\
\vspace{.1in}

Let $\mathbb{P}$ denote the probability and $\mathbb{E}
$ denote the expectation. For the standard definition 
of terms in probability, we refer to E.~Çınlar's book. 
(\cite{Probability})


\subsection{John's Decomposition}

Let $K\subset \mathbb{R}^n$ be a convex body in 
John's position. A point $u\in \mathbb{R}^n$ is a 
contact point of $K$ and $B_2^n$ if $u \in \partial K 
\cap \partial B_2^n$. A classical theorem of F. John 
provides a decomposition of identity in terms of 
contact points.(\cite[p 52]{AGA})


\begin{theorem} \label{JP}
	Let $K$ be a convex body in $\mathbb{R}^n$ that 
	contains $B_2^n$. Then, $K$ is in John's position if 
	and only if there exist contact points $u_1,..,u_m$ 
	and $c_1,..,c_m>0$ such that

	\begin{enumerate}
    \item $\sum_{i=1}^m c_i u_i \otimes u_i=I_n$, and
    \item $\sum_{i=1}^m c_i u_i =\vec{0}$.
    \end{enumerate}

\end{theorem}
	
Therefore, we can check from its contact points whether 
a convex body is in John's position.


\subsection{Measure Concentration on $S^{n-1}$}

Below we include two measure concentration inequalities 
on $S^{n-1}$.

The first inequality is the upper bound for the measure 
of a spherical cap (see, e.g., \cite[p86]{AGA}):

\begin{prop} \label{SphericalCapBound}
	Let $A_t=\{ \theta \in S^{n-1} \, , \, \theta_1>t\}
	$, then $\sigma_{n-1}(A_t)\le 2\exp(-C_3t^2n)$ where 
	$C_3>0$ is a universal constant.
\end{prop}


The second inequality is the concentration inequality 
for Lipschitz continuous functions on the sphere (see, 
e.g., \cite{D}):

\begin{theorem}[Measure Concentration on $S^{n-1}$] 
\label{DL}

	Let $f : S^{n-1} \rightarrow \mathbb{R}$ be a 
	Lipschitz continuous function with Lipschitz 
	constant $b$. Then, for every $t>0$,

	$$
	    \sigma_{n-1}(\{x\in S^{n-1} : | f(x) -\mathbb{E}
	    (f)|\ge bt\})\le 4\exp(-C_4t^2n),
	$$
	where $C_4>0$ is a universal constant.

\end{theorem}


\subsection{Measure Concentration for log-concave 
probability measures}


We include Borell's theorem (\cite[p. 31]{AGA}) and its 
application on comparison of moments (\cite[p. 121]
{AGA}):

\begin{theorem} \label{borell's lemma}

	Let $\mathbb{P}$ be the probability that is 
	uniformly distributed in a convex body $K$. Let $U$ 
	be a closed, convex and symmetric set wherein $
	\mathbb{P}(U)=\delta>1/2$. Then, for any $t>1$, we 
	have
    $$
		\mathbb{P}((tU)^c)\le \delta \left(\frac{1-
		\delta}{\delta}\right)^{\frac{t+1}{2}}.
    $$
    
\end{theorem}

\begin{theorem}	\label{logconcave and seminorm}

	Let $\mu$ be a non-degenerate log-concave 
	probability measure on $\mathbb{R}^n$. If $f:
	\mathbb{R}^n \rightarrow \mathbb{R}$ is a seminorm, 
	then, for any $q>p\ge 1$, we have
    
   $$
		(\mathbb{E}|f|^p)^{1/p}\le  (\mathbb{E}|f|
		^q)^{1/q} \le C_5\frac{q}{p}(\mathbb{E}|f|
		^p)^{1/p},
	$$
	
    where $C_5>0$ is some universal constant.

\end{theorem}

In the end, we include one more theorem about log-concave 
probability measures (Corollary 1 in \cite{On 
the equivalence between geometric and arithmetic means 
for log-concave measures}):

\begin{theorem} \label{logconcave 2}

	For each $0<b<1$ there exists a constant $C_b$ such 
	that for every log-concave probability measure $\mu$ 
	and every measurable convex symmetric set $U$ with $
	\mu(U)=b$ we have
	
    $$	
    	\mu(tU)\le C_b t\mu(U) \mbox{ for } t\in [0,1].
    $$

\end{theorem}


\section{ Proof of the main theorem}

Since the result of Theorem \ref{main theorem} is not 
affected by applying an affine transformation on $K$ or 
$P$, Theorem \ref{main theorem} can be rephrased as 
follows:

\begin{theorem} \label{main theorem John's position}
	
	For a sufficiently large $n\in \mathbb{N}$,
	\begin{enumerate}
	\item There exists a convex body $K \subset 
			\mathbb{R}^n$ in John's position such that  
			$|x_K|\ge (1-C_0\sqrt{\frac{\log(n)}{n}})n$ 
			where $C_0>0$ is a universal constant.
	
	\item There exists a convex polytope $P\subset 
			\mathbb{R}^n$ in John's position with 
			$O(n^2)$ facets such that $|x_P|\ge 
			C_1\frac{n}{\log(n)}$ where $C_1>0$ is a 
			universal constant.  
    \end{enumerate}

\end{theorem}


We write points in the form $x=(y,t)$ where $y\in 
\mathbb{R}^{n-1}$ corresponds to $\{e_i\}_{i=1}^{n-1}$ 
and $t\in \mathbb{R}$ corresponds to $e_n$. For a 
convex body $K$, we write $x_K:=(y_K,t_K)$. Observe 
that, for $R>0$, 

\begin{eqnarray*}
	&	t_K-R &\ge 0\\
	\Leftrightarrow & \frac{1}{\text{vol}(K)}\int_K (t-R) 
	dx&\ge 0\\
	\Leftrightarrow & \int_K (t-R) dx&\ge 0.\\
\end{eqnarray*}
 
Also, for a convex body $K\subset \mathbb{R}^n$, let 
$K_t:= \{y \in \mathbb{R}^{n-1} \, , \, (y,t)\in K\}$, 
which is a slice of the convex body $K$. Let $[a_K,b_K]
$ be the orthogonal projection of $K$ to the span of 
$e_n$.
   
Assuming $0\in K_t$ for all $t\in[a_k,b_k]$, let 
$\rho_K(\cdot ,t)$ denote the radial function of $K_t$ 
as a convex body in $\mathbb{R}^{n-1}$. Then,
 
\begin{eqnarray*}
	 &&\int_{K} (t-R) dx \\
    &=& \int_{a_K}^{b_K} (t-R)\int_{K_t}dydt\\
    &=& \int_{a_K}^{b_K} (t-R)(n-1)\kappa_{n-1}
    \int_0^{\rho_K(\theta,t)}
    \int_{S^{n-2}} r^{n-2}drd\sigma_{n-2}(\theta)dt\\
    &=& \kappa_{n-1}\int_{S^{n-2}}\int_{a_K}^{b_K}
    \rho_K(\theta,t)^{n-1}(t-R)dt
    d\sigma_{n-2}(\theta) ,
\end{eqnarray*}
where $\kappa_n$ denote the volume of $B_2^n$. With $|
x_K|\ge |t_K|$, we conclude

\begin{align}
	\label{MainArgument}
	\begin{split}
	&	\int_{S^{n-2}}\int_{a_K}^{b_K}
	\rho_K(\theta,t)^{n-1}(t-R)dt
    d\sigma_{n-2}(\theta)\ge 0\\ 
	\Rightarrow & |x_K|\ge R.
	\end{split}
\end{align}

Before moving on to the proof of the main theorem, we 
examine two simple convex bodies in $\mathbb{R}
^n$. 

Let $0\in B \subset \mathbb{R}^{n-1}$ be a $n-1$ 
dimensional convex body. We define $B_1,B_2\subset 
\mathbb{R}^n$ as

\begin{eqnarray*}
	B_1&:=& \{ (y,t)\in \mathbb{R}^n \, \, , \, \, y\in 
	B \text{ and } t\in [0,n+1]\} \text{, and }\\
    B_2&:=& \{ (y,t)\in \mathbb{R}^n \, \, , \, \, y\in 
    \frac{t}{n+1}B \text{ and } t\in [0,n+1]\}.\\   
\end{eqnarray*}

In other words, $B_1$ is a cylinder and $B_2$ is a 
cone. Both of them have the same base $B$ and height 
$n+1$. We have $t_{B_1}=\frac{n+1}{2}$. 

For $t_{B_2}$, using the fact that $B_2$ is a cone, 
we have
\begin{eqnarray*}
	t_{B_2}&=& \langle x_{B_2},e_n\rangle \\
	&=&  \frac{1}{\text{vol}(B_2)} \int_{B_2}tdx\\
	&=&	  \frac{n}{(n+1)\text{vol}(B)}\int_{0}
	^{n+1}\text{vol}(B)(\frac{t}{n+1})^{n-1}tdt\\
	&=&  n.
\end{eqnarray*}

Comparing these two examples, we see that $x_{B_2}$ is much closer 
to its base. For the same reason, the convex hull of 
$B_2^n$ and $ne_n$, which is in John's position, has 
a center of mass that lies in $B_2^n$, because its shape is 
similar to that of a cone.\\

We will construct examples in the Theorem \ref{main 
theorem John's position} as the intersection of two convex 
bodies, $Q\cap L$. $Q$ and $L$ will satisfy the 
following:

\begin{enumerate}
	\item $Q$ is in John's position. $L$ contains 
			$B_2^n$. Thus, $Q\cap L$ is also in John's 
			position.
			 
	\item 	$L$ will be a cone (or a cylinder) with the 
			property that $Q\cap L$ and $L$ have a similar 
			shape. Therefore, $x_{Q\cap L}$ behaves like 
			the center of mass of a cone (or a 
			cylinder).
\end{enumerate}



\subsection{Construction of $Q$}

The following proposition is related to the contact 
points decomposition of the identity:

\begin{prop} \label{CJP}
	Let $u_1,...,u_m$ be unit vectors in $\mathbb{R}
	^{n-1}=\mbox{span}\{e_1,..,e_{n-1}\} \subset 
	\mathbb{R}^n$, and $c_1,...,c_m>0$ be some positive 
	numbers such that 
    $$\sum_{i=1}^m c_i u_i\otimes 
	u_i=I_{n-1} \text{\quad and \quad}
    \sum_{i=1}^mc_iu_i=\vec{0}.$$ 
    Set 
	$v_i=(\sqrt{1-\frac{1}{n^2}}u_i,\frac{1}{n})\in 
	\mathbb{R}^n$ for $i=1,...,m$ and 
	$v_0=(\vec{0},-1)$. With $c_i'=\frac{c_i}{1-\frac{1}
	{n^2}}$ and $c_0'=\frac{n}{n+1}$, we obtain

	$$
		\sum_{i=0}^m c'_i v_i\otimes v_i= I_n,
	\text{\quad and \quad} \sum_{i=0}^m c'_iv_i=\vec{0}.$$
\end{prop}

\begin{proof}

From the definition of $v_i$, we have
$$
	v_i\otimes v_i=\frac{1}{n^2}e_n\otimes e_n+
	\frac{1}{n}\sqrt{1-\frac{1}{n^2}}(e_n\otimes u_i + 	u_i \otimes e_n)
    +(1-\frac{1}{n^2})u_i\otimes u_i.
$$

We know that $n-1=\mbox{Tr}(I_{n-1})=\mbox{Tr}
(\sum_{i=1}^m c_i u_i\otimes u_i)=\sum_{i=1}^m c_i$.
Thus, we have
$$
	\sum_{i=0}^m c_i v_i \otimes v_i= \frac{n-1}{n^2}	e_n\otimes e_n 
    +(1-\frac{1}{n^2})I_{n-1}.
$$
where we use the fact that $\sum_{i=1}^m c_i=n-1$ and $
\sum_{i=1}^m c_iu_i=\vec{0}$. Now let $c'_i=\frac{c_i}
{1-\frac{1}{n^2}}=\frac{n^2c_i}{n^2-1}$ for $i=1,...,m$
and $c'_0=\frac{n}{n+1}$. We then have
$$
	\sum_{i=1}^m c'_iv_i\otimes v_i+c'_0(-e_n)\otimes
	(-e_n)=\frac{n-1}{n^2-1}e_n\otimes e_n + I_{n-1}+
	\frac{n}{n+1}e_n\otimes e_n=I_n.
$$
	
Also,

$$
	\sum_{i=1}^m c'_iv_i-c'_0e_n=(\frac{n-1}{n}
	\frac{n^2}{n^2-1}-\frac{n}{n+1})e_n=\vec{0}.
$$

\end{proof}

The points $\{u_j\}_{j=1}^{2(n-1)}=\{\pm e_i\}_{i=1}
^{n-1}$  with $c_j=\frac{1}{2}$ satisfy the 
assumption  of Proposition \ref{CJP}. We set

$$
	A:=\{(\pm\sqrt{1-\frac{1}{n^2}}e_i,\frac{1}{n})\}
	_{i=1}^{n-1}\cup \{(\vec{0},-1)\}
$$

and

$$
	Q:= \{ x \in \mathbb{R}^n \, , \, \forall u \in A \, 
	\, \langle x,u\rangle \le 1 \}.
$$

The set $A$ is the collection of contact points of $Q$. 
By Proposition \ref{CJP} and Theorem \ref{JP}, $Q$ is 
in John's position. \\

Let $B_\infty^{n-1}:=\{ y\in \mathbb{R}^{n-1} \, , \, 
\forall i=1,2,\cdots,n-1 \, \, |\langle y,e_i\rangle|
\le 1\}$ be the unit cube in $\mathbb{R}^{n-1}$.

\begin{eqnarray*}
	Q &=& \{ x \in \mathbb{R}^n \, , \, \forall u \in A 
	\, \, \langle x,u \rangle \le 1 \}\\
	&=& \{ (y,t) \in \mathbb{R}^n \, , \, y\in \frac{n-
	t}{\sqrt{n^2-1}}B_\infty^{n-1} \text{ and } 
	t\in[-1,n]\}.
\end{eqnarray*}

$Q$ is in John's position and it is a cone with 
base $\frac{n+1}{\sqrt{n^2-1}}B_\infty^{n-1}$ and 
height $n+1$. Thus, $Q_t$ is 
$\frac{n-t}{\sqrt{n^2-1}}B_\infty^{n-1}$ for 
$t\in[-1,n]$. Since the radial function of 
$B_\infty^{n-1}$ is $
\rho_{B_\infty^{n-1}}(\theta)=\frac{1}{\max\{|\langle 
\theta,e_i\rangle|\}_{i=1}^{n-1}}$. We have

\begin{equation}
	\label{RadialFunQ}
	\rho_Q(\theta,t)=\frac{1}{\max\{|\langle 
	\theta,e_i\rangle|\}_{i=1}^{n-1}}\frac{n-t}
	{\sqrt{n^2-1}}.
\end{equation}


\subsection{Proof of Theorem \ref{main theorem John's 
position}(1)}

We define
\begin{equation}
	L:=\{(y,t)\in \mathbb{R}^{n} \, , \, y\in (2+
	\frac{t}{n})B_2^{n-1} \text{ and } t\in[-1,n]\}.
\end{equation}

In particular, $L_t$ is equal to $(2+\frac{t}{n})B_2^{n-1}$ and
the raidal function is $\rho_L(\theta,t)=2+\frac{t}{n}$.

Fix $R_0=n-\frac{C_0}{2}\sqrt{\log(n)n}$ for some 
$C_0>0$ that we will determine later. Then, we have $
\rho_L(\theta,R_0)=3-\frac{C_0}{2}\sqrt{\frac{\log(n)}
{n}}$. By (\ref{RadialFunQ}),

$$
	\rho_Q(\theta,R_0)=\frac{1}{\max\{|\langle 
	\theta,e_i\rangle|\}_{i=1}^{n-1}}\frac{C_0}
	{2}\frac{\sqrt{\log(n)n}}{\sqrt{n^2-1}}.
$$ 

We split $S^{n-2}$ into two components by defining
$$
	O_1:=\{\theta \in S^{n-2} , \rho_Q(\theta,R_0)\le 
	\rho_L(\theta,R_0)\}.
$$

For a sufficiently large $n$, we have

\begin{align*}
	O_1&= &\{ \theta\in S^{n-2} \, , \, \frac{1}{(3-
	\frac{C_0}{2}\sqrt{\frac{\log(n)}{n}})}\frac{C_0}
	{2}\frac{\sqrt{\log(n)n}}{\sqrt{n^2-1}}\le  \max\{|
	\langle \theta,e_i\rangle|\}_{i=1}^{n-1}\}\\
	&\subset &\{ \theta\in S^{n-2} \, , \, \frac{C_0}
	{6}\frac{\sqrt{\log(n)n}}{\sqrt{n^2-1}}\le  \max\{|
	\langle \theta,e_i\rangle|\}_{i=1}^{n-1}\}\\
	&\subset& \cup_{i=1}^{n-1}\{\theta \in S^{n-2}, 
	\frac{C_0}{6}\frac{\sqrt{\log(n)n}}{\sqrt{n^2-1}}
	\le |\langle \theta,e_i\rangle|\}.
\end{align*}
 
Due to Proposition \ref{SphericalCapBound}, the measure of $O_1$ can be bounded:
\begin{align*}
	\sigma_{n-2}(O_1)&\le 4n\exp(-\frac{1}{36}
	C_3C^2_0\frac{n^2}{n^2-1}\log(n))\\
	&\le 4\exp\left((1-\frac{1}{36}C_3C^2_0)\log(n)
	\right).
\end{align*}

By setting $C_0:=\sqrt{\frac{72}{C_3}}$, for a
sufficiently large $n$, we have 

\begin{equation}
	\label{O_1bound}
	\sigma_{n-2}(O_1) \le 4\exp(-\log(n))\le 
	\frac{1}{2}.	
\end{equation}

Moreover, $\rho_L(\theta,t)$ is increasing with respect 
to $t\in[-1,n]$, while $\rho_Q(\theta,t)$ is decreasing 
with respect to $t\in[-1,n]$. We may conclude that,

\begin{equation}
	\label{O_1^cbound}
	\forall\theta \in O_1^c \text{ , }\rho_Q(\theta,t)
	\ge \rho_L(\theta,t) \text{ for }t\in[-1, R_0].
\end{equation}

We define $K$ to be the intersection of $Q$ and $L$, $K=Q\cap L$. Then, we have $K_t=Q_t\cap L_t$ and thus $
\rho_K(\theta,t)=\min\{\rho_Q(\theta,t),
\rho_L(\theta,t)\}$. 

By (\ref{MainArgument}), it is sufficient
to prove 
\begin{equation}
\label{MainArgument1}
\int_{S^{n-1}}\int_{-1}^n \rho_K(\theta,t)^{n-1}(t-
	R)dtd\sigma_{n-2}(\theta)\ge 0 ,
\end{equation}
with $R=n-C_0\sqrt{\log(n)n}$. For the inner integral in
(\ref{MainArgument1}):

\begin{align*}\begin{split}
	&\int_{-1}^n \rho_K(\theta,t)^{n-1}(t-R)dt\\
\le & \int_{-1}^{R_0} \rho_K(\theta,t)^{n-1}(t-R)dt\\
= & -\int_{-1}^R \rho_K(\theta,t)^{n-1}(R-t)dt + 
\int_R^{R_0} \rho_K(\theta,t)^{n-1}(t-R)dt.\end{split}
\end{align*}

For the first component, with $\rho_K(\theta,t)\le
\rho_L(\theta,t)=2+\frac{t}{n}$, we have 
$$
	\int_{-1}^R \rho_K(\theta,t)^{n-1}(R-t)dt \le\int_{-1}^R 
    (2+\frac{t}{n})^{n-1}(R-t)dt.
$$
The integral on the right side is computable via integration by parts:
\begin{align*}
    & \int_{-1}^R (2+\frac{t}{n})^{n-1}(R-t)dt\\
	= & \left.(2+\frac{t}{n})^n(R-t) \right]_{-1}^R+
	\int_{-1}^R (2+\frac{t}{n})^ndt\\
	= & -(2-\frac{1}{n})^n(R+1)+ \frac{n}{n+1}(2+
	\frac{R}{n})^{n+1}-\frac{n}{n+1}(2-\frac{1}{n})^{n+1}\\
	\le & \frac{n}{n+1}(2+\frac{R}{n})^{n+1}.
\end{align*}
Thus, 
\begin{equation} \label{minus1}
	\int_{-1}^R\rho_K(\theta,t)^{n-1}(R-t)dt\le 
    \frac{n}{n+1}(2+\frac{R}{n})^{n+1}.
\end{equation}

For $\theta \in O_1^c$, due to (\ref{O_1^cbound}) we have 
$\rho_K(\theta,t)=\rho_L(\theta,t)=(2+\frac{t}{n})$ for 
$t\in[-1,n]$. Thus, we have the equality when $\theta\in O_1^c$:
$$
\int_R^{R_0} \rho_K(\theta,t)^{n-1}(t-R)dt
	=\int_R^{R_0} \rho_L(\theta,t)^{n-1}(t-R)dt.
$$
Again, the integral on the right side is computable: 
\begin{align*}\begin{split}
	&\int_R^{R_0} (2+\frac{t}{n})^{n-1}(t-R)dt \\
	=&\left.(2+\frac{t}{n})^n(t-R) \right]_{R}^{R_0}-
	\int_{R}^{R_0} (2+\frac{t}{n})^ndt\\
	=& (2+\frac{R_0}{n})^n(R_0-R)-\frac{n}{n+1}(2+
	\frac{R_0}{n})^{n+1}+\frac{n}{n+1}(2+\frac{R}
	{n})^{n+1}.
\end{split}
\end{align*}

Observe that, for a sufficiently large $n$, we have 
$$
	(R_0-R)=\frac{C_1}{2}\sqrt{\log(n)n}>4>2\frac{n}
	{n+1}(2+\frac{R_0}{n}).
$$
Hence, the previous equality can be bounded:
\begin{eqnarray*}
	&(2+\frac{R_0}{n})^n(R_0-R)-\frac{n}{n+1}(2+
	\frac{R_0}{n})^{n+1}+\frac{n}{n+1}(2+\frac{R}
	{n})^{n+1}\\
    \ge& \frac{1}{2}(2+\frac{R_0}{n})^n(R_0-R).
\end{eqnarray*}
We conclude that, for any $\theta \in O_1^c$,
\begin{equation}
\label{plus1}
\int_R^{R_0} \rho_K(\theta,t)^{n-1}(t-R)dt \ge 
\frac{1}{2}(2+\frac{R_0}{n})^n(R_0-R).
\end{equation}

Now we can derive the main inequality (\ref{MainArgument1}). 
First, we split the integral:

\begin{eqnarray*}
	&\int_{S^{n-1}}\int_{-1}^n \rho_K(\theta,t)^{n-1}(t-
	R)dtd\sigma_{n-2}(\theta)\\
	=& \int_{S^{n-1}}\int_{-1}
	^{R}\rho_K(\theta,t)^{n-1}(t-R)dtd\sigma_{n-2}(\theta)\\
    &+ \int_{S^{n-1}}\int_{R}
	^{R_0}\rho_K(\theta,t)^{n-1}(t-R)dtd\sigma_{n-2}(\theta)\\
    &+ \int_{S^{n-1}}\int_{R_0}
	^{n}\rho_K(\theta,t)^{n-1}(t-R)dtd\sigma_{n-2}
	(\theta).
\end{eqnarray*}
By (\ref{minus1}), the first summand satisfies
$$
	\int_{S^{n-1}}\int_{-1}
	^{R}\rho_K(\theta,t)^{n-1}(t-R)dtd\sigma_{n-2}(\theta)\ge 
    -\frac{n}{n+1}(2+\frac{R}{n})^{n+1}.
$$
According to (\ref{plus1}) and (\ref{O_1bound}), the second 
summand satisfies
\begin{eqnarray*}
&&\int_{S^{n-1}}\int_{R}
	^{R_0}\rho_K(\theta,t)^{n-1}(t-R)dtd\sigma_{n-2}(\theta)\\
&\ge& \int_{O_1^c}\int_{R}
	^{R_0}\rho_K(\theta,t)^{n-1}(t-R)dtd\sigma_{n-2}(\theta)\\
&\ge& \frac{1}{4}(2+\frac{R_0}{n})^n(R_0-R).
\end{eqnarray*}
Noticing that the third summand is non-negative, we conclude that 
\begin{eqnarray*}
&\int_{S^{n-1}}\int_{-1}^n \rho_K(\theta,t)^{n-1}(t-
	R)dtd\sigma_{n-2}(\theta)\\
\ge& -\frac{n}{n+1}(2+\frac{R}{n})^{n+1}
+\frac{1}{4}(2+\frac{R_0}{n})^n(R_0-R).
\end{eqnarray*}

With $\frac{1}{4}(R_0-R)>2>\frac{n}{n+1}(2+\frac{R}{n})
$ and $(2+\frac{R_0}{n})^n>(2+\frac{R}{n})^n$, we get 

$$
-\frac{n}{n+1}(2+\frac{R}{n})^{n+1}+\frac{1}{4}(2+
\frac{R_0}{n})^n(R_0-R)>0
$$
for a sufficiently large $n$. Hence,
$$
\int_{S^{n-1}}\int_{-1}^n \rho_K(\theta,t)(t-
R)dtd\sigma_{n-2}(\theta)>0.
$$ 

We conclude from (\ref{MainArgument}) that 
$$
	|x_K|>R=n-C_0\sqrt{\log(n)n}=(1-
	C_0\sqrt{\frac{\log(n)}{n}})n.
$$


\subsection{Proof of Theorem \ref{main theorem John's 
position} (2)}

To construct $P$ in Theorem \ref{main theorem John's 
position} (2) we define a cylinder $L_2$, 
which is the intersection of $O(n^2)$ number of half 
spaces and set $P:=Q\cap L_2$, where $Q$ is the same as above. \\

\vspace{.1in}

Let $\{\epsilon_n\}$ be a decreasing sequence. Later we 
will specify $\epsilon_n$, but for now we assume that

\begin{eqnarray}
	\label{epsilonconstraint} \frac{10}{n}<\epsilon_n<1 
	&\text{, and}\\
	\label{epsilondecay} \lim_{n\rightarrow +\infty}
	\epsilon_n=0.
\end{eqnarray}

Let
$$
	A':=\{ \pm(1-\epsilon_n)e_i\pm\sqrt{1-(1-
	\epsilon_n)^2}e_j \}_{ i,j<n \,i\neq j},
$$
and
$$
	L_2 := \{ (y,t) \in \mathbb{R}^n, \langle y, u 
	\rangle\le 1 \, \, \forall u\in A' \text{ and } 
	t\in[-1,n]\}.
$$

We have $|A'|=4n(n-1)$ and $L_2$ is a cylinder with 

$$
	L_{2,t} = \{ y \in \mathbb{R}^{n-1}, \langle y, u 
	\rangle\le 1 \, \, \forall u\in A'\}
$$
for $t\in [-1,n]$. Let $P=Q\cap L_2$. Since 
$B_2^n\subset L_2$ and $Q$ is in John's position, $P$ 
is in John's position. Following the same approach from 
the proof of Theorem \ref{main theorem John's position} 
(1), we want to show

\begin{equation}
	\label{EQ1} \int_{S^{n-2}}\int_{-1}^n
	\rho_P(\theta,t)^{n-1}(t-\frac{1}{5}\epsilon_n n)dt 
	d\sigma_{n-2}(\theta) >0.
\end{equation}

Then, we can conclude $|x_P|>\frac{1}{5}\epsilon_n n$.

\vspace{.1in}

For convenience, let $Q':=Q_{\epsilon_n n}$ and 
$L':=L_{2,\epsilon_n n}$. Also, let $\rho_{Q'}
(\cdot):=\rho_{Q}(\cdot,\epsilon_n n)$ and $\rho_{L'}
(\cdot):=\rho_{L_2}(\cdot,\epsilon_n n)$. We will show 
that for the majority of $\theta \in S^{n-2}$, $
\rho_P(\theta,t)=\rho_{L_2}(\theta,t)$ for $t \in 
[-1,\epsilon_n n]$. In the case that $\rho_P(\theta,t)
\neq\rho_{L_2}(\theta,t)$ for some $t$ in $
[-1,\epsilon_n n]$, $\rho_P(\theta,t)$ will be nicely 
bounded.


\begin{prop} \label{PropertyofKandL}

	With the notation above, let 
	$$
		O_2:=\{\theta \in S^{n-2} \, , \,\rho_{Q'}
		(\theta)\le \rho_{L'}(\theta)\}.
	$$

	For a sufficiently large $n$, we have
    \begin{enumerate}

	    \item $\forall \theta \in O_2$, $\rho_{Q'}
	    (\theta)\le 4\sqrt{\epsilon_n n}$.
   		 \item $\sigma_{n-2}(O_2)\le 4n\exp(-\frac{C_6}
   		 {\epsilon_n})$, where $C_6>0$ is a universal 
   		 constant.
    \end{enumerate}
\end{prop}

\begin{proof}
Let $y \in \partial Q' \cap L'$. Then, there exists 
$i$ such that $|y_i|=(1-\epsilon_n)\frac{n}
{\sqrt{n^2-1}}=\rho_{Q'}(\frac{y}{|y|})$. 
Following the conditions from the definition of 
$L'$, we have, for $j\neq i$,

\begin{eqnarray*}
	&(1-\epsilon_n)|y_i|+\sqrt{1-(1-\epsilon_n)^2}|y_j|
	&\le 1\\
	\Rightarrow & \sqrt{1-(1-\epsilon_n)^2}|y_j| &\le  
	1-(1-\epsilon_n)^2\\
	\Rightarrow & |y_j| &\le  \sqrt{1-(1-\epsilon_n)^2},
\end{eqnarray*}

where for the second inequality we use $\frac{n}
{\sqrt{n^2-1}}\ge 1$.\\
    
From the previous argument, $y \in \partial Q' \cap L'$ 
implies that

\begin{equation}
	\label{normof|y|}
	|y|\le \sqrt{(n-2)(1-(1-\epsilon_n)^2)+(1-
	\epsilon_n)^2\frac{n^2}{n^2-1}} .
\end{equation}

By (\ref{epsilonconstraint}) and (\ref{epsilondecay}), 
we have $0<(1-(1-\epsilon_n)^2)=2\epsilon_n-
\epsilon_n^2 \le 2\epsilon_n$ and $n\epsilon_n>1$. Hence, 
(\ref{normof|y|}) becomes

\begin{eqnarray*}
	|y|\le & 	\sqrt{(n-2)(1-(1-\epsilon_n)^2)+(1-
	\epsilon_n)^2\frac{n^2}{n^2-1}}\\
	\le & \sqrt{2\epsilon_n n+2}\\
	\le & 2\sqrt{\epsilon_n n},
\end{eqnarray*}
which proves Claim (1) in Proposition \ref{PropertyofKandL}.\\

For $\theta \in O_2$, $\rho_Q(\theta)\theta \in 
\partial Q'\cap L'$. There exists $i$ such that
$|(\rho_{Q'}(\theta)\theta)_i|=(1-\epsilon_n)\frac{n}
{\sqrt{n^2-1}}$. By (\ref{epsilondecay}),
$(1-\epsilon_n)\frac{n}{\sqrt{n^2-1}}>\frac{1}{2}$ for 
large $n$. 

Together with $\rho_{Q'}(\theta)\le 2\sqrt{\epsilon_n 
n}$,

\begin{equation}
\label{normboundKandL} 
	|\theta_i|=\frac{(1-\epsilon_n)\frac{n}
	{\sqrt{n^2-1}}}{\rho_{Q'}(\theta)}\ge\frac{1}
	{2\rho_{Q'}(\theta)}\ge \frac{1}{4\sqrt{\epsilon_n 
	n}}.
\end{equation}

Thus, inequality (\ref{normboundKandL}) leads to the 
following inclusion:

$$
	O_2 \subset \cup_{i=1}^{n-1} \{ \theta\in S^{n-2} 
	\, , \, |\theta_i|\ge \frac{1}{4\sqrt{\epsilon_n n}} 
	\}.
$$

By Proposition \ref{SphericalCapBound},
$$
	\sigma_{n-2}(\{ \theta \, , \, |\theta_i|\ge  
	\frac{1}{4\sqrt{\epsilon_n n}} \})\le 4\exp(-
	\frac{C_3}{16\epsilon_n}\frac{n-1}{n})\le 4\exp(-
	\frac{C_6}{\epsilon_n}).
$$ 
	
Therefore, using the union bound, we conclude that 
$$
	\sigma_{n-2}(O_2)\le 4n\exp(-\frac{C_6}
	{\epsilon_n}).
$$
\end{proof}

\begin{prop} \label{rhoLislarge}
	With the notation above, there exists a constant \
	$C_7>0$ such that if the sequence $\{\epsilon_n\}$ 
	satisfies $\frac{C_7}{\log(n)}>\epsilon_n$ for a large 
	sufficiently $n$, then

	$$
		\sigma_{n-2}( \{\theta \, , \, \rho_{L'}(\theta) 
		\le 5\sqrt{\epsilon_n n} \}) \le  4\exp(-
		\frac{C_8}{\epsilon_n}),
	$$
	where $C_8>0$ is a universal constant.
\end{prop}

\begin{proof}
Let $\|\cdot\|$ be the norm on $\mathbb{R}^{n-1}$ such 
that $L'$ is the unit ball that corresponds to the norm 
$\|\cdot\|$. More specifically, for $y\in \mathbb{R}
^{n-1}$,

$$
	\|y\| = \max_{1\le i,j<n \, , \, i\neq j}\{(1-
	\epsilon_n)|y_i|+\sqrt{1-(1-\epsilon_n)^2}|y_j| \}.
$$

Let $g=(g_1,g_2,..,g_{n-1})$ be the standard Gaussian 
random vector in $\mathbb{R}^{n-1}$. Then,

\begin{eqnarray*}
	\mathbb{E}\|g\|&=&\mathbb{E}\max_{1\le i,j<n \, , \,
	 i\neq j}\{(1-\epsilon_n)|g_i|+\sqrt{1-(1-
	 \epsilon_n)^2}|g_j| \}\\
    &\le& 2 \mathbb{E}\max_{i=1,..,n-1}|g_i| \le 
    c'\sqrt{\log(n)},
\end{eqnarray*}
where $c'>0$ is a universal constant and the last 
inequality is a classical result  for the extreme value 
of independent Gaussian random variables.

Using the standard polar integration, we obtain the 
following inequality,
$$ 
	\int_{S^{n-2}} \|\theta\|d\sigma_{n-2}(\theta)
	\le\frac{c''}{\sqrt{n}}\mathbb{E}\|g\| ,
$$
where $c''>0$ is a universal constant. Thus, $
\mathbb{E}_{\sigma_{n-2}}\|\theta\|\le 
c'c''\sqrt{\frac{\log(n)}{n}}$. Moreover, $\sup_{\theta 
\in S^{n-2}}\|\theta\|\le 1$ due to the fact that 
$B_2^{n-1}\subset L'$. Therefore, the function $\theta 
\rightarrow \|\theta\|$ is 1-Lipschitz on 
$S^{n-2}$. We set $C_7>0$ to be small enough so that $
\frac{1}{2}\frac{1}{5\sqrt{\epsilon_n n}}
>c'c''\sqrt{\frac{\log(n)}{n}}$. Since $\rho_{L'}
(\theta)=\frac{1}{\|\theta\|}$, we have the equality 
$$
\{ \theta \in S^{n-2} \, , \, 
    \rho_{L'}(\theta) \le 5\sqrt{\epsilon_n n}\}
    =\{ \theta \in S^{n-2} \, , \,
    \|\theta\| \ge \frac{1}{5\sqrt{\epsilon_n n}}).
$$
Furthermore, the inequality 
$\mathbb{E}_{\sigma_{n-2}}\|\theta\|\le \frac{1}{2}\frac{1}{5\sqrt{\epsilon_n n}}$
implies
$$
\{ \theta \in S^{n-2} \, , \,
    \|\theta\| \ge \frac{1}{5\sqrt{\epsilon_n n}})
 \subset \{\theta \in S^{n-2} \, , \,
    \left|\|\theta\|-\mathbb{E}\|\theta\|\right|>
    \frac{1}{10\sqrt{\epsilon_n n}})\}.
$$
Together with Theorem \ref{DL}, we may conclude that
\begin{eqnarray*}
    &&\sigma_{n-2}(\{ \theta \in S^{n-2} \, , \, 
    \rho_{L'}(\theta) \le 5\sqrt{\epsilon_n n}\})\\
    &\le& \sigma_{n-2}(\{\theta \in S^{n-2} \, , \,
    \left|\|\theta\|-\mathbb{E}\|\theta\|\right|>
    \frac{1}{10\sqrt{\epsilon_n n}})\} \\
    &\le & 4\exp(-\frac{C_8}{\epsilon_n}),
\end{eqnarray*}
where we use Theorem \ref{DL} in the last inequality.
\end{proof}
	

Now we are able to prove Theorem \ref{main theorem 
John's position} (2).

\begin{proof}[Proof of Theorem \ref{main theorem John's 
position} (2)]

We want to choose $\epsilon_n$ so that
\begin{equation}
	\label{EQ2} \sigma_{n-2}(O_2)<\frac{1}{4}  ,
\end{equation}
and

\begin{equation}
	\label{EQ3} \sigma_{n-2}(\{\theta\in S^{n-2}, 
	\rho_{L'}(\theta)\le 5\sqrt{\epsilon_n n}\})\le 
	\frac{1}{4},
\end{equation}
for a large $n$. 

According to Proposition \ref{PropertyofKandL}, the 
first condition can be achieved if $\epsilon_n 
<\frac{c}{\log(n)}$ for some $c>0$ when $n$ is large.

Moreover, we also want to choose $\epsilon_n<\frac{c'}
{\log(n)}$ so that we can apply Proposition 
\ref{rhoLislarge} to get $\sigma_{n-2}(\{\theta\in 
S^{n-2}, \rho_{L'}(\theta)\le 5\sqrt{\epsilon_n n}\})
\le \frac{1}{4}$. Therefore, we can set $
\epsilon_n=\frac{c''}{\log(n)}$ for some $c''>0$ so 
that (\ref{EQ2}) and (\ref{EQ3}) hold.\\

\vspace{.1in}

Recall that from (\ref{EQ1}) our goal is to show that
$$
	\int_{S^{n-2}}\int_{-1}^n\rho_{P}(\theta,t)^{n-1}(t-
	\frac{1}{5}\epsilon_n n)dt d\sigma_{n-2}(\theta)>0.
$$

Since $P=Q\cap L_2$,
$$
	\rho(\theta,t)=\min\{\rho_{Q}(\theta,t),\rho_{L_2}
	(\theta,t)\}=\min\{\rho_{Q}(\theta,t),\rho_{L'}
	(\theta)\}.
$$

We handle the inner integral differently for $\theta 
\in O_2$ and $\theta \notin O_2$.

\begin{itemize}
	\item In the case that  $\theta \notin O_2$:\\
   	First, we have $\rho_{Q}(\theta,\epsilon_n n) \ge
    \rho_{L_2}(\theta,\epsilon_n n)$. Thus,
    $\rho_P(\theta,t)=\rho_{L'}(\theta)$ for $t \in 
    [-1,\epsilon_n n]$. This is because $\rho_{L'}
    (\theta)$ is a constant and $\rho_P(\theta,t)$ is 
    decreasing with respect to $t$.  Thus,
    \begin{eqnarray*}
       &&\int_{-1}^n\rho_P(\theta,t)^{n-1}(t-\frac{1}
       {5}\epsilon_n n)dt\\
       &\ge&\int_{-1}^{\epsilon_n n}
       \rho_P(\theta,t)^{n-1}(t-\frac{1}{5}\epsilon_n
       n)dt\\
       &=& \int_{-1}^{\epsilon_n n}\rho_{L'}
       (\theta)^{n-1}(t-\frac{1}{5}\epsilon_n n)dt.\\
    \end{eqnarray*}
    We split the integral to two parts:
\begin{eqnarray*}
&\int_{-1}^{\epsilon_n n}\rho_{L'}
       (\theta)^{n-1}(t-\frac{1}{5}\epsilon_n n)dt\\
       =&
  \int_{-1}^{2\frac{1}{5}\epsilon_n
       n+1}\rho_{L'}(\theta)^{n-1}(t-\frac{1}
       {5}\epsilon_n n)dt\\
       &+\int_{2\frac{1}{5}\epsilon_n 
       n+1}^{\epsilon_n n}\rho_{L'}(\theta)^{n-1}(t-
       \frac{1}{5}\epsilon_n n)dt.\\
\end{eqnarray*}
Due to the symmetry of the integrand with respect to 
$t=\frac{1}{5}\epsilon_n n$, the first summand is $0$. For the 
second summand, we have
\begin{eqnarray*}
       &&\int_{\frac{2}{5}\epsilon_n n+1}^{\epsilon_n
       n}\rho_{L'}(\theta)^{n-1}(t-\frac{1}
       {5}\epsilon_n n)dt\\
       &\ge& (\epsilon_n n-\frac{2}{5}\epsilon_n 
       n-1)\rho_{L'}(\theta)^{n-1}(\frac{1}
       {5}\epsilon_n n+1)\\
       &\ge& \frac{(\epsilon_n n)^2}{10}\rho_{L'}
       (\theta)^{n-1},
     \end{eqnarray*}
     where in the second to last inequality we used 
     that $\frac{2}{5}\epsilon_n n +1\le \frac{1}
     {2}\epsilon_n n$ by (\ref{epsilonconstraint}).
	We conclude that
\begin{equation}
	\label{O_2^c}
\forall \theta\in O_2^c \, ,    
\int_{-1}^n\rho_P(\theta,t)^{n-1}(t-\frac{1}{5}\epsilon_n n)dt \ge \frac{(\epsilon_n n)^2}{10}\rho_{L'}
       (\theta)^{n-1}.
\end{equation}
\vspace{.1in}
%
%
     \item  In the case that $\theta \in O_2$:\\
     From Proposition \ref{PropertyofKandL}, we know 
     that $\rho_{Q}(\theta,\epsilon_n n)\le 
     2\sqrt{\epsilon_n n}$. Therefore, since $
     \rho_{Q}(\theta, t)$ is linear on $[-1,n]$ and $
     \rho_{Q}(\theta,n)=0$, we see that for any 
     $t\in[-1,n]$,
            
     $$
     		\rho_{Q}(\theta,t)\le \frac{n+1}{n-\epsilon_n 
     		n}2\sqrt{\epsilon_n n} \le 4\sqrt{\epsilon_n 
     		n},
     $$
     for a sufficiently large $n$. We have 
     
     \begin{eqnarray*}
	     \int_{-1}^n\rho_P(\theta,t)^{n-1}(t-\frac{1}
	     {5}\epsilon_n n)dt
	     \ge\int_{-1}^{\frac{1}{5}\epsilon_n
	     n}\rho_P(\theta,t)^{n-1}(t-\frac{1}
	     {5}\epsilon_n n)dt,
     \end{eqnarray*}
     because the integrand is positive for $t>\frac{1}{5}\epsilon_n n$. 
     Then, using the estimate of $\rho_Q(\theta,t)\le 4\sqrt{\epsilon_n n}$, 
   
     \begin{eqnarray*}
     &&\int_{-1}^{\frac{1}{5}\epsilon_n
	     n}\rho_P(\theta,t)^{n-1}(t-\frac{1}
	     {5}\epsilon_n n)dt\\
	     &\ge&-(\frac{1}{5}\epsilon_n n+1)
	     (4\sqrt{\epsilon_n n})^{n-1}(1+\frac{1}
	     {5}\epsilon_n n) \\
        &\ge& -\frac{4}{25}(\epsilon_n n)^2 
        (4\sqrt{\epsilon_n n})^{n-1},
     \end{eqnarray*}
      where in the last inequality we used
      $\frac{1}{5}\epsilon_n n+1\le \frac{2}
      {5}\epsilon_n n$, which is valid for a large $n$.
      Therefore, we have
      \begin{equation}
      \label{O_2}
\forall \theta\in O_2 \, , \,     
\int_{-1}^n\rho_P(\theta,t)^{n-1}(t-\frac{1}
	     {5}\epsilon_n n)dt \ge 
         -\frac{4}{25}(\epsilon_n n)^2 
        (4\sqrt{\epsilon_n n})^{n-1}.
      \end{equation}
\end{itemize}

Now we are able to derive the main inequality.
\begin{eqnarray}   
\label{mainstep1}\begin{split}
	&&\int_{S^{n-2}}\int_{-1}^n\rho_P(\theta,t)^{n-1}(t-
	\frac{1}{5}\epsilon_n n)dt d\sigma_{n-2}(\theta)\\
   &=& \int_{O_2}\int_{-1}^n\rho_P(\theta,t)^{n-1}(t-
   \frac{1}{5}\epsilon_n n)dt d\sigma_{n-2}(\theta)\\
   &&+\int_{O_2^c}\int_{-1}^n\rho_P(\theta,t)^{n-1}(t-
   \frac{1}{5}\epsilon_n n)dtd\sigma_{n-2}(\theta).
   \end{split}
\end{eqnarray}

Applying (\ref{O_2^c}), the second summand satisfies
$$
\int_{O_2^c}\int_{-1}^n\rho_P(\theta,t)^{n-1}(t-
   \frac{1}{5}\epsilon_n n)dtd\sigma_{n-2}(\theta)
\ge  (\epsilon_n n)^2 \int_{O_2^c}\frac{1}
   {10}\rho_{L'}(\theta)^{n-1}d\sigma_{n-2}(\theta).
$$

Let $U:=\{\theta \, , \, \rho_{L'}(\theta) \ge 
5\sqrt{\epsilon_n n}\}$. From (\ref{EQ2}) and 
(\ref{EQ3}) we know that $\sigma_{n-2}(U\cap O_2^c)\ge 
\frac{1}{2}$ for a large $n$. Since the integrand is positive,
$$
\int_{O_2^c}\frac{1}{10}\rho_{L'}(\theta)^{n-1}d\sigma_{n-2}(\theta)\ge 
\int_{U\cap O_2^c}\frac{1}{10}\rho_{L'}(\theta)^{n-1}d\sigma_{n-2}(\theta).
$$
Thus, 
$$
\int_{O_2^c}\int_{-1}^n\rho_P(\theta,t)^{n-1}(t-
   \frac{1}{5}\epsilon_n n)dtd\sigma_{n-2}(\theta)
   \ge (\epsilon_n n)^2 \int_{U\cap O_2^c}\frac{1}{10}\rho_L(\theta)^{n-1}d\sigma_{n-2}(\theta).
$$

For the first summand of (\ref{mainstep1}), we apply
(\ref{O_2}) and  (\ref{EQ2}) to get 
\begin{eqnarray*}
&&\int_{O_2}\int_{-1}^n\rho_P(\theta,t)^{n-1}(t-
   \frac{1}{5}\epsilon_n n)dt d\sigma_{n-2}(\theta)\\
&\ge&   -(\epsilon_n n)^2\sigma_{n-2}(O_2)\frac{4}{25}
   (4\sqrt{\epsilon_n n})^{n-1}.\\
\end{eqnarray*}

Combining the inequalities for the two summands together we have

\begin{eqnarray*}
&&\int_{S^{n-2}}\int_{-1}^n\rho_P(\theta,t)^{n-1}(t-
	\frac{1}{5}\epsilon_n n)dt d\sigma_{n-2}(\theta)\\
   &\ge&(\epsilon_n n)^2 [\frac{1}{20}
   (5\sqrt{\epsilon_n n})^{n-1}-\frac{1}{25}
   (4\sqrt{\epsilon_n n})^{n-1}]\\
    &\ge& 0.
\end{eqnarray*}
Therefore, the center 
of mass is at least $C_1\frac{n}{\log(n)}$ away from 
$0$, where $C_1:=\frac{c''}{5}$.

\end{proof}


\section{The relation between the conjectures}

Let $K\subset \mathbb{R}^n$ be a convex body in John's 
position and $X$ be a random vector uniformly 
distributed in $K$. Let $M_K$ denote the median of $|X|
$, which is the unique value satisfying

$$
	\mathbb{P}( |X|\le M_K)=\frac{1}{2}.
$$

\begin{lemma} \label{moment and expectation}
	Let $K\subset \mathbb{R}^n$ be a convex body. Let 
	$X$ be a random vector uniformly distributed in $K$. 
	Let $M_K$ denote the median of $|X|$. Then, we have
   
   $$
    	\frac{M_K}{\sqrt{2}} \le (\mathbb{E}|X|^2)^{1/2} \le 
    	C_9 M_K,
   $$
   where $C_9>0$ is a universal constant.
\end{lemma}

\begin{proof}

The first inequality is standard:
$$
	\mathbb{E}|X|^2\ge \mathbb{E}(|X|^21_{|X|\ge M_K})\ge \frac{1}{2}M_K^2
$$
Thus, the first inequality can be obtained by taking square root on both sides.

To prove the second 
one, let $R$ be the number such that $\mathbb{P}(|X|\le 
R)=\frac{2}{3}$. We can apply Theorem \ref{borell's 
lemma} with $U=R B_2^n$ and $\delta=\frac{2}{3}$ to get

$$
 	\mathbb{P}( |X|>t R)\le \frac{\sqrt{2}}{3} 2^{-t/2} 
 	\mbox{ for } t>1.
$$

A simple integration shows that
     
$$
	\mathbb{E}|X|^2 \le c R^2,
$$

for a universal constant $c>0$. \\

Now we apply Theorem \ref{logconcave 2} with 
$b=\frac{2}{3}$ and $U=R B_2^n$ to obtain

$$
  	\mathbb{P}( |X|\le M_KB_2^N)) \le C_b \frac{M_K}{R} 
  	\mathbb{P}( |X|\le R),
$$

which implies that $M_K\ge \frac{3}{C_{2/3}}R$. 

\end{proof}

We could also relate $\mathbb{E}|X|^2$ and the center 
of mass of $K$, $x_K$, when $K$ is in John's position. 

\begin{lemma} \label{second moment and center of mass}
	There exists $C_{10},C_{11}>0$ such that, for any 
	convex body $K \subset \mathbb{R}^n$ in John's 
	position, we have
   $$
    	|x_K|^2\le \mathbb{E}|X|^2\le  C_{10}|x_K|
    	^2+C_{11}n,
	$$
    where $X$ is a random vector uniformly distributed 
    in $K$.
\end{lemma}

This result was proved by M.~Fradelizi, G.~Paouris and 
C.~Schütt in \cite{ Simplices in the euclidean ball}. 

Here we present a different proof.

\begin{proof}

Since $K$ is in John's position, there exists $\{u_i\}
_{i=1}^m\subset S^{n-1}$ and $\{c_i\}_{i=1}^m$ with 
$c_i>0$ such that $\sum_{i=1}^m c_iu_i=0$ and $
\sum_{i=1}^m c_i u_i \otimes u_i=I_n$.

In particular, $$\mathbb{E}|X|^2=\mathbb{E}\sum_{i=1}^m 
c_i(\langle X,u_i\rangle)^2.$$ Also, $|x_K|
^2=\sum_{i=1}^m c_i (\langle x_K,u_i\rangle)^2$.

Given that $u_i$ is a contact point of $K$, we have $
\langle x,u_i\rangle\le 1$ for all $x\in K$. As a 
consequence, with $\langle x_K,u_i\rangle=\mathbb{E}
\langle X,u_i\rangle$ we have

$$ 
	0\le \mathbb{E}|\langle X,u_i\rangle|-|\langle 
	x_K,u_i\rangle| \le 2.
$$

Here, the first inequality follows from Jensen's 
inequality while the second one relies on an elementary  
observation that for any random variable $Y \le 1$, $
\mathbb{E} |Y| =2 \mathbb{E} \max\{Y,0\} -\mathbb{E} Y
\le 2+ |\mathbb{E} Y|$. Thus, 
$$
	|\langle x_K,u_i\rangle|^2\le (\mathbb{E}|\langle
    X,u_i\rangle|)^2\le 3(|\langle x_K,u_i\rangle|
    ^2+2).
$$

According to Theorem \ref{logconcave and seminorm}, we 
have

$$
   \mathbb{E}|\langle X,u_i\rangle|\le (\mathbb{E}|
   \langle X,u_i\rangle|^2)^{1/2}\le 2C_5 \mathbb{E}|
   \langle X,u_i\rangle|.
$$

Therefore, we can conclude that
$$
 	|x_K|^2\le \mathbb{E}|X|^2\le 12C_5(|x_K|^2+2\sum 
 	c_i)\le C|x_K|^2+C'n,
$$
where the last inequality uses the fact that $
\sum_{i=1}^m c_i=n$.

\end{proof}

\begin{cor} \label{relation between conj 1 and conj 2}
	Conjecture \ref{conj2} and Conjecture \ref
    {conj1} are equivalent.
\end{cor}

\begin{proof}
	Let $K\subset \mathbb{R}^n$ be a convex body. 
	Since the result is invariant under affine 
	transformations, we may assume that $K$ is in John's 
	position. Let $X$ be a random vector uniformly 
	distributed in $K$ and $M_K$ be the median of the 
	random variable $|X|$. \\
	
	Suppose Conjecture \ref{conj1} is true. There exists 
	a universal constant $C>0$ such that $|x_K|\le 
	C\sqrt{n}$. According to Lemma \ref{moment and 
	expectation} and Lemma \ref{second moment and center 
	of mass}, 

	\begin{eqnarray*}
		M_K & \le & \sqrt{2}(\mathbb{E}|X|^2)^{1/2}\\
		& \le &  \sqrt{2}\sqrt{C_{10}|x_K|^2+C_{11}n}\\
		& \le &  \sqrt{2n}\sqrt{C_{10}C^2+C_{11}}.
	\end{eqnarray*}
	
	This argument is valid for any convex body $K$; 
	therefore, Conjecture \ref{conj2} is true. \\

	On the other hand, assuming Conjecture \ref{conj2} 
	is valid, there exists a universal constant $C>0$ such 
	that $M_K \le C\sqrt{n}$. Again, according to Lemma 
	\ref{moment and expectation} and Lemma \ref{second 
	moment and center of mass}, 
	
	\begin{eqnarray*}
     |x_K| &\le & (\mathbb{E}|X|^2)^{1/2}\\
     &\le & 	C_9M_K\\
     &\le & C_9C\sqrt{n}.
	\end{eqnarray*}

	Therefore, Conjecture \ref{conj1} is true. 
\end{proof}

The examples in Corollary \ref{main cor} will be 
examples $K,P$, which are constructed in Theorem \ref{main theorem 
John's position}. For Corollary \ref{main cor}(2), the 
result will follow by $|x_P|\le C_9M_P$. Corollary 
\ref{main cor}(1) is a more delicate situation, and so 
the same argument does not apply. Observe that, for $R>0$, 
$$K\cap RB_2^{n} \subset K\cap \{x\in \mathbb{R}^n, \langle x,e_1\rangle \le R \}.$$

It is sufficient to show a stronger statement:
$$
\text{vol}( K\cap \{x\in \mathbb{R}^n, \langle x,e_n\rangle \le R \}) \le \text{vol}( K\cap \{x\in \mathbb{R}^n, \langle x,e_1\rangle > R \})
$$
for $R=n-C_0'\sqrt{\log(n)n}$. 
Adapting the notations from the proof, this is equivalent to show 
$$
	\int_{S^{n-1}}\int_{-1}^n 
	\rho_K(\theta,t)^{n-1}\text{sign}(t-
	R)dtd\sigma_{n-2}>0.
$$ The proof of this statement is almost 
identical to the proof of Theorem \ref{main 
theorem John's position}(1).

\end{flushleft}

\section{Acknowledgement}
I thank my advisor, Professor Mark Rudelson, for 
discussing and providing advice about this problem. I 
also thank Professor Santosh Vempala for identifying 
the problem and explaining its computer science 
background.

\end{document}